\documentclass{article}

\usepackage{amsmath,amscd,amsfonts,amssymb,amsthm}
\usepackage{graphicx}
\usepackage{epstopdf}
\usepackage{a4wide}
\usepackage{graphicx}
\usepackage{subfigure}
\usepackage[latin1]{inputenc}
\usepackage{color}
\usepackage{cite}

\newcommand{\subj}[1]{\par\noindent{\bf Mathematics Subject Classification 2010: }#1.}
\newcommand{\keyw}[1]{\par\noindent{\bf Keywords: }#1.}

\theoremstyle{definition}

\newtheorem{theorem}{Theorem}
\newtheorem{example}{Example}
\theoremstyle{remark}
\newtheorem{remark}{Remark}

\def\a{\alpha}
\def\t{\triangle T}
\def\DS{\displaystyle}
\def\LI{{_aI_t^\a}}
\def\RI{{_tI_b^{\a}}}
\def\LRD{{_aD_t^\a}}
\def\RRD{{_tD_b^{\a}}}
\def\LCD{{^C_aD_t^\a}}
\def\RCD{{^C_tD_b^\a}}
\def\a{\alpha}
\def\t{\tau}
\def\C{\DS\left(^{n+m-\a}_{\quad k}\right)}
\def\DS{\displaystyle}

\begin{document}

\title{A numerical method to solve higher-order fractional differential equations}

\author{Ricardo Almeida$^1$\\
{\tt ricardo.almeida@ua.pt}
\and
 Nuno R. O. Bastos$^{1,2}$\\
{\tt nbastos@estv.ipv.pt}}

\date{$^1$Center for Research and Development in Mathematics and Applications (CIDMA)\\
Department of Mathematics, University of Aveiro, 3810--193 Aveiro, Portugal\\
$^2$Department of Mathematics, School of Technology and Management of Viseu \\
Polytechnic Institute of Viseu\\
3504--510 Viseu, Portugal}

\maketitle


\begin{abstract}

In this paper, we present a new numerical method to solve fractional differential equations. Given a fractional derivative of arbitrary real  order, we present an approximation formula for the fractional operator that involves integer-order derivatives only. With this, we can rewrite FDEs in terms of a classical one and then apply any known technique. With some examples, we show the accuracy of the method.

\end{abstract}

\subj{26A33, 49M25, 49M25}

\keyw{fractional calculus, fractional differential equations, approximation formulas, numerical methods}


\section{Introduction}

Since the beginning of differential calculus, the question of what could be a derivative of non-integer order was pertinent and Leibniz himself wondered about the derivative of order $\a=1/2$. Liouville carried out a serious investigation on the subject, and presented a notion of fractional integrator operator. Later, starting with Cauchy's formula for an $n$-fold integral, Riemann defined a fractional integration as it is known today, and become a basis for the fractional calculus theory. A different type of fractional operators have appeared in 1967 \cite{Caputo}, due to Michel Caputo, and as proven to be applicable in many situations. It reveals two advantages: the derivative of a constant is zero and, when solving fractional differential equations involving this operator, it is not necessary to define fractional order initial conditions and we may consider ordinary ones.
In the last years, FDEs have revealed to model better some real phenomena, since these fractional operators contain memory and from experimental data, some dynamics of trajectories are modeled by non-integer order derivatives. Because of this, they have found numerous applications in various research areas and engineering applications (viscoelasticity, viscoplasticity, modeling polymers, transmission of ultrasound waves, etc) \cite{Alikhanov,Bhalekar,Chen,Das,Diethelm,Ezzat,Johnny,Leung,Nerantzaki,Mirjana,Zuomao}. However, there is no effective and easy-to-use method to solve such differential equations. For this reason, we find in the literature a vast number of numerical methods in order to be able to solve them \cite{Butera,Galeone,Gulsu,Jafari,Li,Mueller,Pedas}.

We begin in Section \ref{sec:FC} with a short introduction to fractional calculus, as presented in e.g. \cite{hilfer,Kilbas,Oldham,Podlubny}. In Section
\ref{Sec:theory}, we present and prove the main result of the paper: under some smoothness assumptions, we can approximate a fractional derivative of arbitrary real order by a sum that involves integer-order derivatives only. An estimation for the error is also given. With this, we extend the main results of \cite{Atan0,Atan2,almeida1}, by considering fractional derivatives of arbitrary real order. In Section \ref{Sec:examples}, we present some examples; first we test the efficiency of the method, by comparing the exact expression of the Caputo fractional derivative of a given function with some numerical approximations. At the end, we exemplify how it can be useful to solve fractional differential equations.

\section{Preliminaries}\label{sec:FC}

Let us review some necessary definitions on fractional calculus. Let $x:[a,b]\rightarrow\mathbb{R}$ be a function, $\alpha$ a positive non-integer number and $n\in\mathbb N$ be such that $\a\in(n-1,n)$.
In what follows, we assume that $x$ is sufficiently good in order to the fractional operators be well defined.
The left and right Riemann--Liouville fractional integrals of order $\a$ is a generalization of the Cauchy's formula to arbitrary real numbers, and are defined as
$$\LI x(t) =\frac{1}{\Gamma(\alpha)}\int_a^t (t-\tau)^{\alpha-1}x(\tau) d\tau,$$
and
$$\RI x(t)=\frac{1}{\Gamma(\alpha)}\int_t^b (\tau-t)^{\alpha-1}x(\tau) d\tau,$$
respectively. For fractional derivatives, we consider two types of operators.
The left and right Riemann--Liouville fractional derivatives are given by
$$\LRD x(t) =\frac{1}{\Gamma(n-\alpha)}\frac{d^n}{dt^n}\int_a^t (t-\tau)^{n-\alpha-1}x(\tau) d\tau,$$
and
$$\RRD x(t)=\frac{(-1)^n}{\Gamma(n-\alpha)}\frac{d^n}{dt^n}\int_t^b (\tau-t)^{n-\alpha-1}x(\tau) d\tau,$$
respectively.
The left and right Caputo fractional derivatives are given by
$$\LCD x(t) =\frac{1}{\Gamma(n-\alpha)}\int_a^t (t-\tau)^{n-\alpha-1}x^{(n)}(\tau) d\tau,$$
and
$$\RCD x(t)=\frac{(-1)^n}{\Gamma(n-\alpha)}\int_t^b (\tau-t)^{n-\alpha-1}x^{(n)}(\tau) d\tau,$$
respectively. There exists a relation between these two fractional derivatives, to know:

\begin{equation}\label{relation1}\LCD x(t)=\LRD x(t)-\sum_{k=0}^{n-1}\frac{x^{(k)}(a)}{\Gamma(k-\alpha+1)}(t-a)^{k-\alpha},\end{equation}
and
\begin{equation}\label{relation2}\RCD x(t)=\RRD x(t)-\sum_{k=0}^{n-1}\frac{x^{(k)}(b)}{\Gamma(k-\alpha+1)}(b-t)^{k-\alpha}.\end{equation}
Therefore, if
$$x(a)=x'(a)=\ldots=x^{(n-1)}(a)=0 \, \Rightarrow \, \LCD x(t)=\LRD x(t),$$
and if
$$x(b)=x'(b)=\ldots=x^{(n-1)}(b)=0 \, \Rightarrow \, \RCD x(t)=\RRD x(t).$$

Although in this paper we deal with the Caputo fractional derivatives, using relations \eqref{relation1} and \eqref{relation2}, similar formulas can be deduced for the Riemann--Liouville fractional derivatives.

Immediate calculations lead to the following. For the power functions
$$x(t)=(t-a)^{\beta-1} \quad \mbox{and} \quad y(t)=(b-t)^{\beta-1},$$
with $\beta>n$, we have
$$\LCD x(t)=\frac{\Gamma(\beta)}{\Gamma(\beta-\a)}(t-a)^{\beta-\a-1},$$
and
$$\RCD y(t)=\frac{\Gamma(\beta)}{\Gamma(\beta-\a)}(b-t)^{\beta-\a-1}.$$
If $x\in C^n[a,b]$, then the Caputo fractional derivatives $\LCD x(t)$ and $\RCD x(t)$ exist and are continuous on $[a,b]$. Moreover,  ${^C_aD_t^\a}x(t)=0$ at $t=a$, and ${^C_tD_b^\a}x(t)=0$ at $t=b$. Also, the Caputo fractional differentiation and the Riemann--Liouville fractional integration can be seen as inverse operations of each other. In fact, if $x\in C[a,b]$, then
$$\LCD \, \LI x(t)=\RCD \, \RI x(t)=x(t),$$
and if $x\in C^n[a,b]$, then
$$\LI \, \LCD x(t)=x(t)-\sum_{k=0}^{n-1}\frac{x^{(k)}(a)}{k!}(t-a)^k,$$
and
$$\RI \, \RCD x(t)=x(t)-\sum_{k=0}^{n-1}\frac{(-1)^kx^{(k)}(b)}{k!}(b-t)^k.$$

\section{Theoretical results}\label{Sec:theory}

In the following, we use the extension of the binomial formula to real numbers:
$$\DS\left(\begin{array}{c}
\gamma\\
k
\end{array}\right)(-1)^k=\frac{\Gamma(k-\gamma)}{\Gamma(-\gamma)k!}.$$

\begin{theorem}\label{teo1} Let $m \in \mathbb N \cup\{0\}$, $\mathbb N \ni N \geq m+1$ and  $x:[a,b]\rightarrow\mathbb{R}$ be a function of class $C^{n+m+1}$. Define
\begin{equation*}
\begin{split}
A_k&=\DS\frac{1}{\Gamma(n+k+1-\a)}\left[1+\sum_{p=m-k+1}^N\frac{\Gamma(p+\a-n-m)}{\Gamma(\a-n-k)(p-m+k)!}\right],\quad \mbox{for} \, k \in \{0,1,\ldots,m\},\\
B_k&=\DS\frac{\Gamma(k+\a-n-m)}{\Gamma(n-\a)\Gamma(\a+1-n)(k-m-1)!},\quad \mbox{for} \, k \in \{m+1,m+2,\ldots,N\},\\
V_k(t)& \DS=\int_a^t (\t-a)^{k}x^{(n)}(\t)d\t,\quad \mbox{for} \, k \in \{0,1,\ldots,N-m-1\},\, t\in[a,b].
\end{split}
\end{equation*}
Then, the following holds:
$$\LCD x(t)=\sum_{k=0}^{m}A_k(t-a)^{n+k-\a} x^{(n+k)}(t)+\sum_{k=m+1}^N B_k(t-a)^{n+m-k-\a}V_{k-m-1}(t)+E_N(t),$$
with
$$|E_{N}(t)|\leq \max_{\tau\in[a,t]}|x^{(n+m+1)}(\tau)| (t-a)^{n+m+1-\a}
\frac{\exp((n+m-\a)^2+n+m-\a)}{\Gamma(n+m+1-\a)N^{n+m-\a}(n+m-\a)}.$$
\end{theorem}

\begin{proof} Starting with the formula
$$\LCD x(t) =\frac{1}{\Gamma(n-\alpha)}\int_a^t (t-\tau)^{n-\alpha-1}x^{(n)}(\tau) d\tau,$$
and integrating by parts with $u'=(t-\tau)^{n-\alpha-1}$ and $v=x^{(n)}(\tau)$, we get
$$\LCD x(t) =\frac{x^{(n)}(a)}{\Gamma(n+1-\alpha)}(t-a)^{n-\a}+\frac{1}{\Gamma(n+1-\alpha)}\int_a^t (t-\tau)^{n-\alpha}x^{(n+1)}(\tau) d\tau.$$
Repeating this process $m$ more times, obtain the formula
$$\LCD x(t) =\sum_{k=0}^m\frac{x^{(n+k)}(a)}{\Gamma(n+k+1-\alpha)}(t-a)^{n+k-\a}+\frac{1}{\Gamma(n+m+1-\alpha)}\int_a^t (t-\tau)^{n+m-\alpha}x^{(n+m+1)}(\tau) d\tau.$$
Using the Taylor's expansion, we obtain the next sum
$$\begin{array}{ll}
(t-\tau)^{n+m-\alpha}& =\DS (t-a)^{n+m-\alpha}\left(1-\frac{\t-a}{t-a}\right)^{n+m-\alpha}\\
& =\DS (t-a)^{n+m-\alpha}\sum_{k=0}^N \C (-1)^k\frac{(\t-a)^k}{(t-a)^k}+\overline E_N(t,\t),
\end{array}$$
where
$$\overline E_N(t,\t)=(t-a)^{n+m-\alpha}\sum_{k=N+1}^\infty \C (-1)^k\frac{(\t-a)^k}{(t-a)^k}.$$
Then,
$$\LCD x(t) =\sum_{k=0}^m\frac{x^{(n+k)}(a)}{\Gamma(n+k+1-\alpha)}(t-a)^{n+k-\a}$$
$$+\frac{(t-a)^{n+m-\a}}{\Gamma(n+m+1-\alpha)}\sum_{k=0}^N \frac{\Gamma(k+\a-n-m)}{\Gamma(\a-n-m)k!(t-a)^k}\int_a^t (\t-a)^k x^{(n+m+1)}(\tau) d\tau+E_{N}(t),$$
where
$$E_{N}(t)=\frac{1}{\Gamma(n+m+1-\alpha)}\int_a^t \overline E_{N}(t,\t) x^{(n+m+1)}(\tau) d\tau.$$
If we split the sum into $k=0$ and the remaining terms $k=1,\ldots, N$, and integrating by parts with $u=(\t-a)^k$ and $v'=x^{(n+m+1)}(\tau)$, we get
$$\LCD x(t) =\sum_{k=0}^{m-1}\frac{x^{(n+k)}(a)}{\Gamma(n+k+1-\alpha)}(t-a)^{n+k-\a}+A_m(t-a)^{n+m-\a} x^{(n+m)}(t)$$
$$+\frac{(t-a)^{n+m-1-\a}}{\Gamma(n+m-\alpha)}\sum_{k=1}^N \frac{\Gamma(k+\a-n-m)}{\Gamma(\a+1-n-m)(k-1)!(t-a)^{k-1}}\int_a^t (\t-a)^{k-1} x^{(n+m)}(\tau) d\tau+E_{N}(t).$$
Repeating the procedure $m$ more times, we obtain
$$\LCD x(t)=\sum_{k=0}^{m}A_k(t-a)^{n+k-\a} x^{(n+k)}(t)+\sum_{k=m+1}^N B_k(t-a)^{n+m-k-\a}V_{k-m-1}(t)+E_N(t).$$
Now we get the upper bound formula for the error $E_N(t)$.
Using the following relations:
$$\left|\frac{\t-a}{t-a}\right|\leq1, \quad \forall \t\in[a,t],$$
and
$$\sum_{k=N+1}^\infty \left|\C (-1)^k\right|\leq \sum_{k=N+1}^\infty \frac{\exp((n+m-\a)^2+n+m-\a)}{k^{n+m+1-\a}}$$
$$\leq\int_N^\infty \frac{\exp((n+m-\a)^2+n+m-\a)}{k^{n+m+1-\a}}\, dk= \frac{\exp((n+m-\a)^2+n+m-\a)}{N^{n+m-\a}(n+m-\a)},$$
the formula is deduced.
\end{proof}
We stress out that for all $t\in[a,b]$, $E_N(t)$ goes to zero as $N$ goes to $\infty$. Also, using Eq. \eqref{relation1}, a similar formula can be deduced for the left Riemann--Liouville fractional derivative. An approximation formula for the right Caputo fractional derivative is straightforward.

\begin{theorem} Let $m \in \mathbb N \cup\{0\}$, $\mathbb N \ni N \geq m+1$ and  $x:[a,b]\rightarrow\mathbb{R}$ be a function of class $C^{n+m+1}$. Define
\begin{equation*}
\begin{split}
A_k&=\DS\frac{(-1)^{n+k}}{\Gamma(n+k+1-\a)}\left[1+\sum_{p=m-k+1}^N\frac{\Gamma(p+\a-n-m)}{\Gamma(\a-n-k)(p-m+k)!}\right], \quad \mbox{for} \, k \in \{0,1,\ldots,m\},\\
B_k&=\DS\frac{(-1)^n\Gamma(k+\a-n-m)}{\Gamma(n-\a)\Gamma(\a+1-n)(k-m-1)!},\quad \mbox{for} \, k \in \{m+1,m+2,\ldots,N\},\\
W_k(t)& \DS=\int_t^b (b-\t)^{k}x^{(n)}(\t)d\t,\quad \mbox{for} \, k \in \{0,1,\ldots,N-m-1\},\, t\in[a,b].
\end{split}
\end{equation*}
Then, the following holds:
$$\RCD x(t)=\sum_{k=0}^{m}A_k(b-t)^{n+k-\a} x^{(n+k)}(t)+\sum_{k=m+1}^N B_k(b-t)^{n+m-k-\a}W_{k-m-1}(t)+E_N(t),$$
with
$$|E_{N}(t)|\leq \max_{\tau\in[t,b]}|x^{(n+m+1)}(\tau)| (b-t)^{n+m+1-\a}
\frac{\exp((n+m-\a)^2+n+m-\a)}{\Gamma(n+m+1-\a)N^{n+m-\a}(n+m-\a)}.$$
\end{theorem}

\begin{remark} Other methods exist in the literature to solve higher-order fractional problems. The most common procedure is to discretize the fractional operator, while our method allows us to rewrite the fractional problem into an ordinary one, and after it we can choose any available technique known in the literature to solve it. For a reference on a numerical method to approximate the fractional derivative of higher-order, we mention \cite{Sousa} and is the following. Let $\alpha$ be such that $1<\alpha<2$ and $x:[a,b]\to\mathbb R$ be a function of class $C^2$. Consider the mesh points defined by
$$t_j = a + j\Delta t,~ j= 0, 1 . . . , N$$
where $\Delta t$ denotes the uniform space step. So, the approximation is the following:
\begin{equation}\label{sousa}_a^C D_t^{\alpha}x(t_j)\approx \frac{\Delta t^{-\alpha}}{\Gamma(3-\alpha)}
\sum_{k=0}^{j-1}d_{j,k}\left(x\left(t_{k+2}\right)-2x\left(t_{k+1}\right)+x\left(t_{k}\right)\right),\end{equation}
where
$$d_{j,k}=(j-k)^{2-\alpha}-(j-k-1)^{2-\alpha}.$$
\end{remark}

\section{Numerical examples}\label{Sec:examples}

We exemplify in this section the purposed formula. All the required computations are executed in Matlab, using a grid on time $x_1,\ldots,x_G$. The error that appears in such approximations is measured by the $L^2$ norm:
\begin{equation}
\label{error:L2}
E(x,y)=\left(\sum_{i=1}^{G}(x_i-y_i)^2\right)^{\dfrac12}.
\end{equation}
For simplicity, we will consider always $G=100$.

\begin{example}
We compare the Caputo fractional derivative of $x(t)=t^6$ and $y(t)=(1-t)^6$,  $t\in[0,1]$, with fractional orders $\a=1.5$ and $\beta=2.5$. The exact expressions are given by
$$\begin{array}{ll}
\DS{^C_0D_t^{1.5}}x(t)=\frac{6!}{\Gamma(5.5)}t^{4.5}, & \quad \DS{^C_0D_t^{2.5}}x(t)=\frac{6!}{\Gamma(4.5)}t^{3.5},\\
\DS{^C_tD_1^{1.5}}y(t)=\frac{6!}{\Gamma(5.5)}(1-t)^{4.5}, & \quad\DS{^C_tD_1^{2.5}}y(t)=\frac{6!}{\Gamma(4.5)}(1-t)^{3.5}.\\
\end{array}$$
For the numerical approximation given by Theorem \ref{teo1}, we consider two distinct cases. First, we fix $m=1$ and take $N\in\{10,15,25,50\}$ (Figure \ref{m=1}); then, we fix $N=50$ and take $m\in\{1,2,3\}$ (Figure \ref{N=50}).

As expected, as $N$ increases, we obtain a better approximation for the fractional derivative.
From Figure \ref{N=50} we can see that, even for a small value of $m$, we already obtain a good approximation for each function.

\begin{figure}[h!]
  \begin{center}
    \subfigure[$\DS{^C_0D_t^{1.5}}x(t)$]{\label{fig_1_1}\includegraphics[scale=0.5]{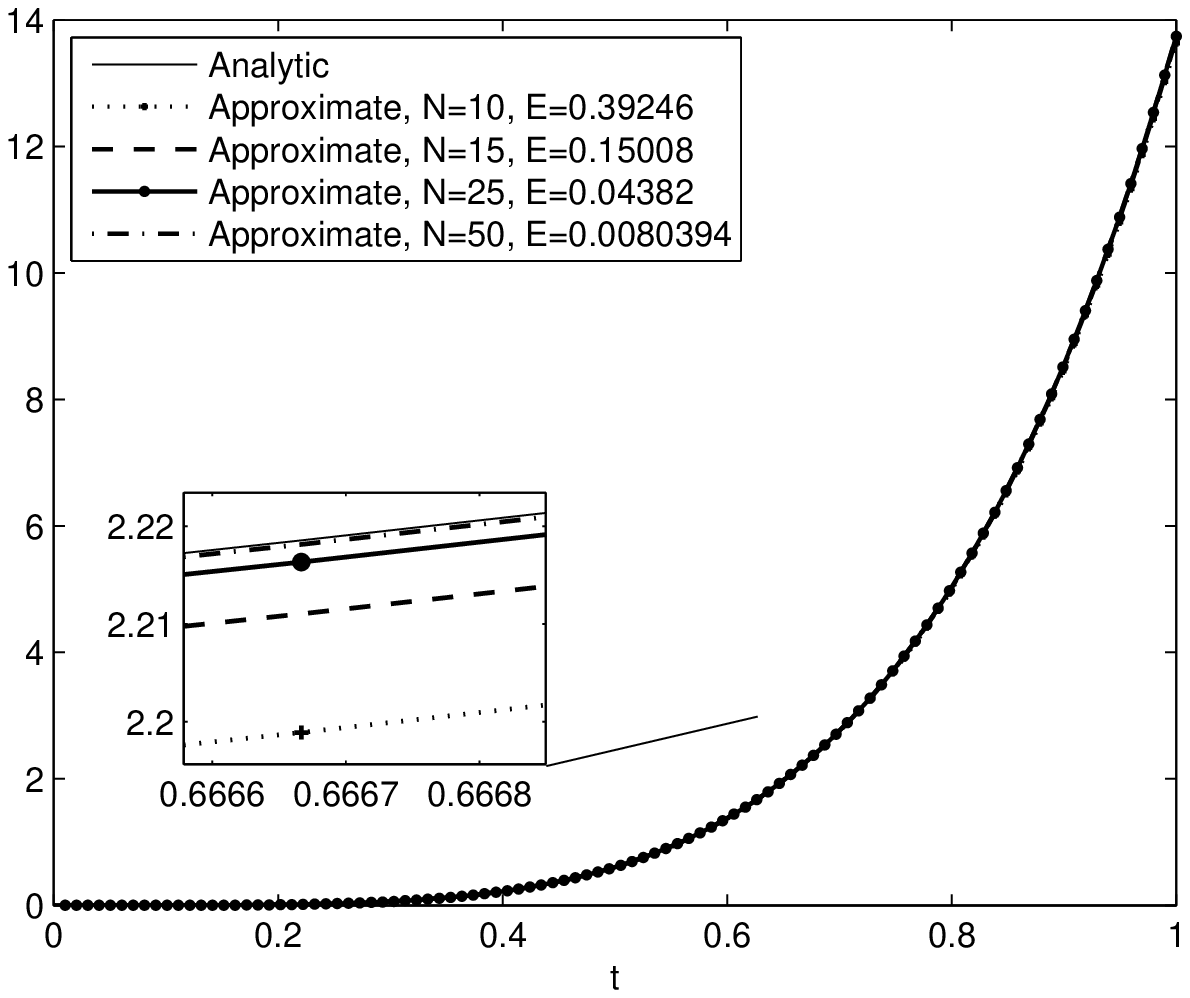}}
    \subfigure[$\DS{^C_0D_t^{2.5}}x(t)$]{\label{fig_1_2}\includegraphics[scale=0.5]{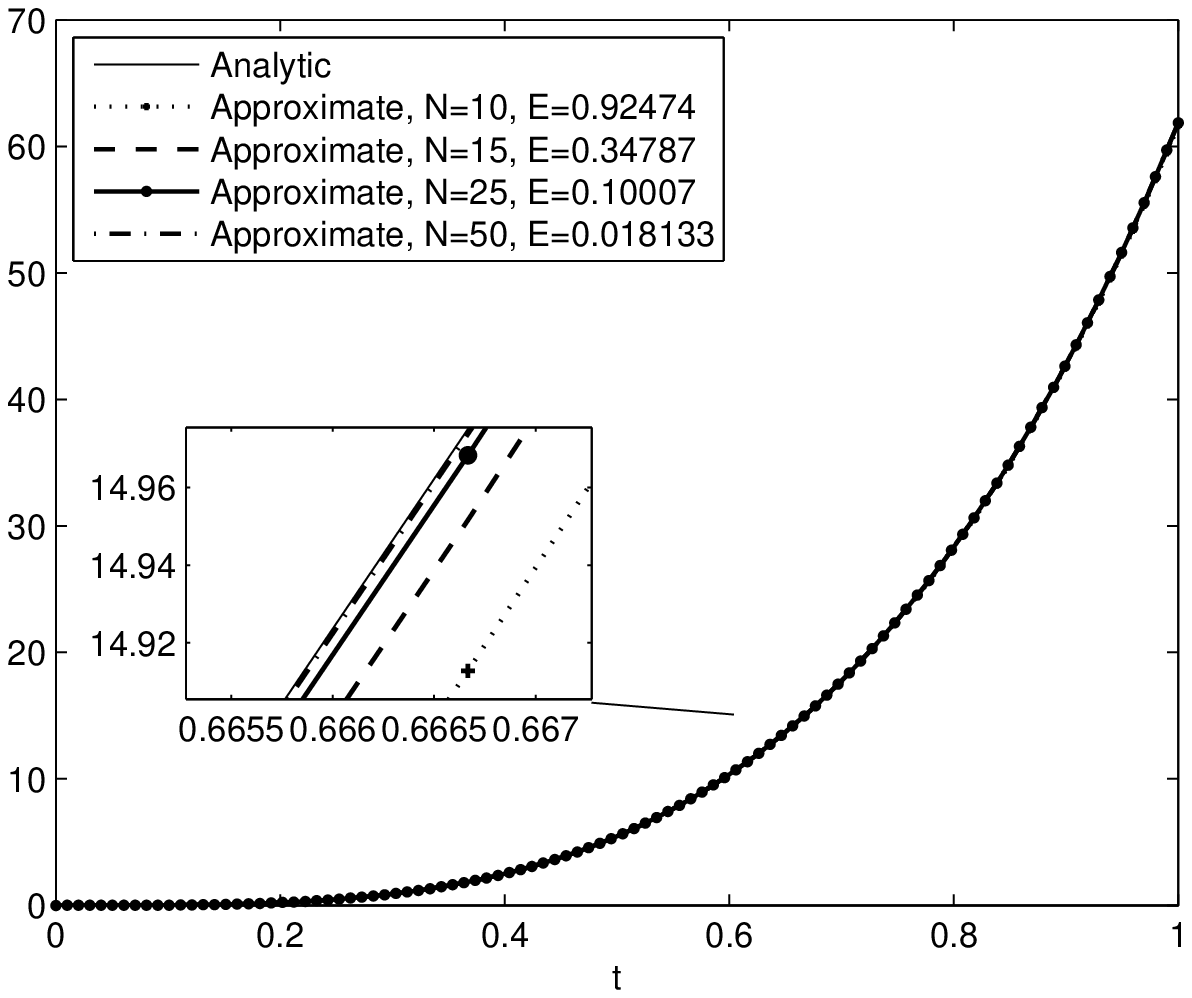}}
        \subfigure[$\DS{^C_tD_1^{1.5}}y(t)$]{\label{fig_1_3}\includegraphics[scale=0.5]{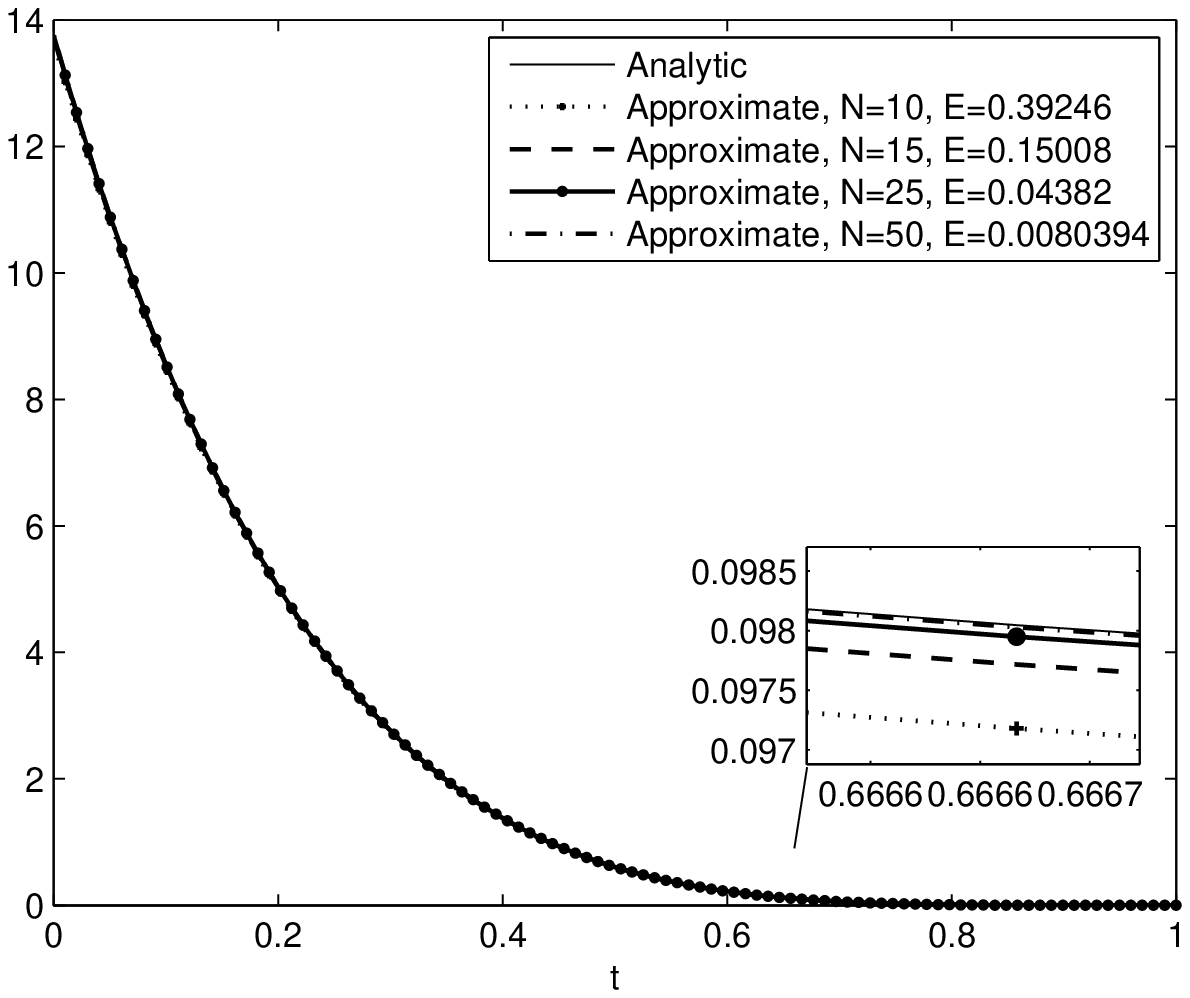}}
        \subfigure[$\DS{^C_tD_1^{2.5}}y(t)$]{\label{fig_1_4}\includegraphics[scale=0.5]{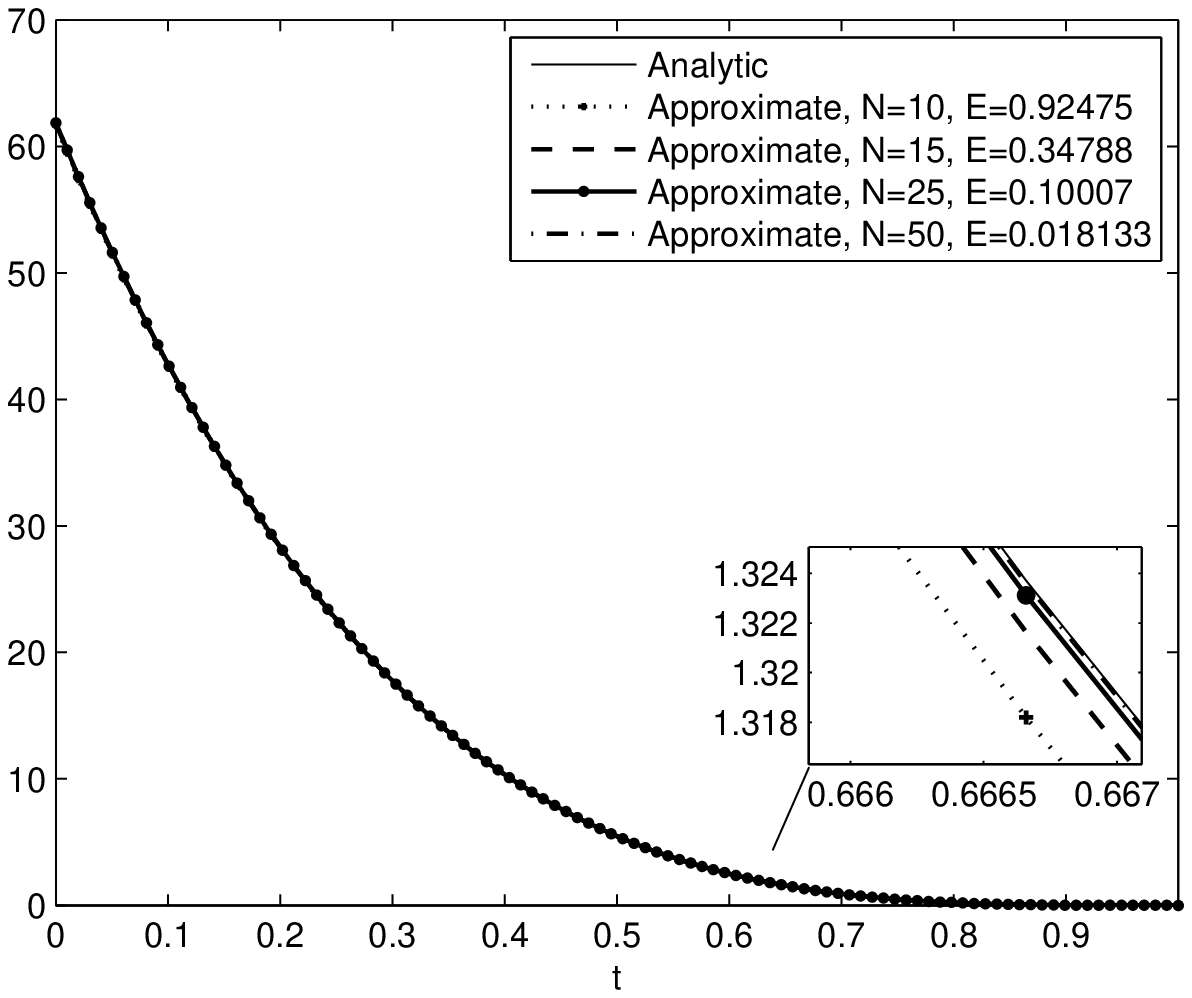}}
      \end{center}
  \caption{For $m=1$: analytic vs. numerical approximation.}\label{m=1}
  \label{m_1}
\end{figure}

\begin{figure}[h!]
  \begin{center}
    \subfigure[$\DS{^C_0D_t^{1.5}}x(t)$]{\label{fig_1_1_N_Fixed}\includegraphics[scale=0.5]{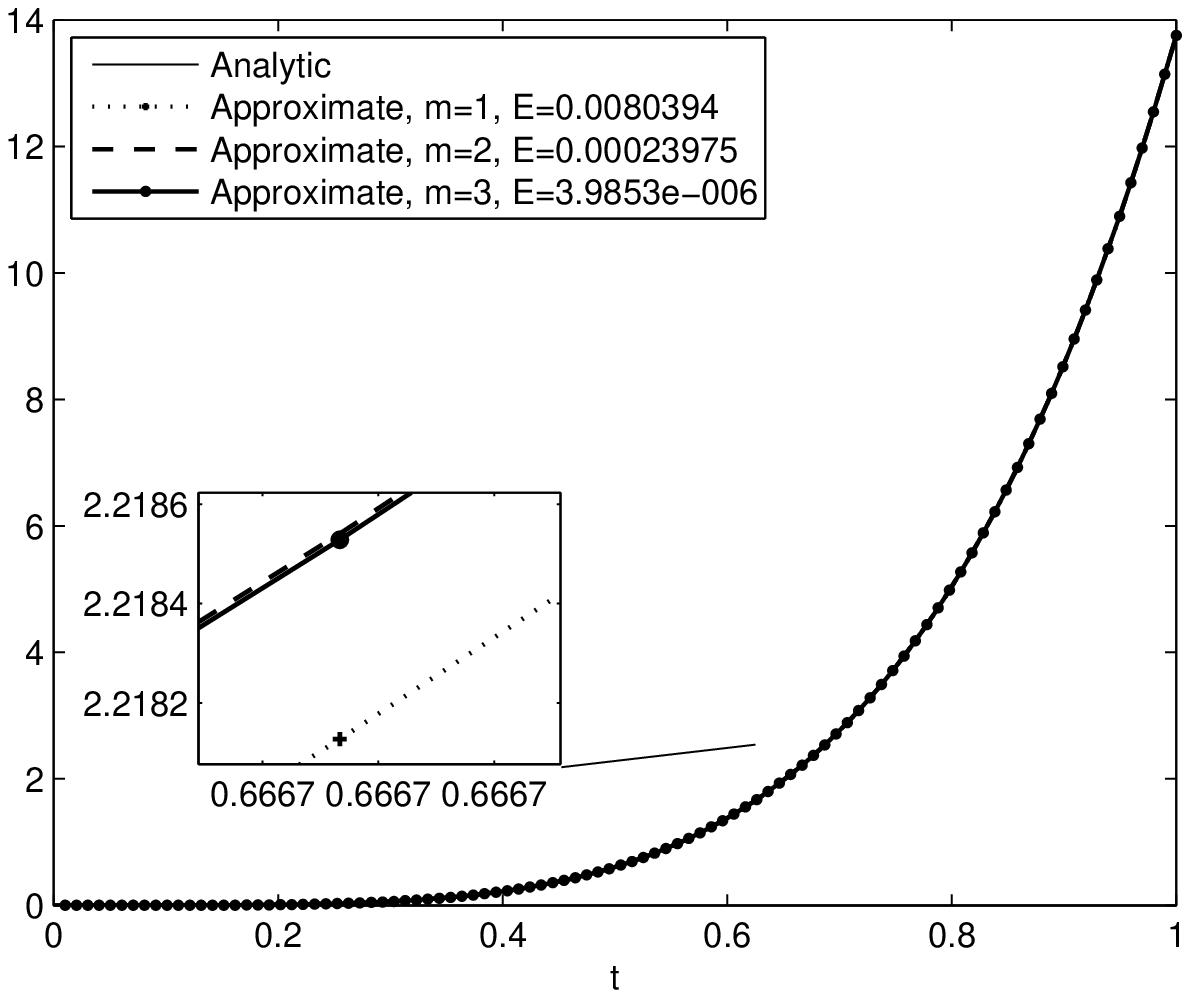}}
    \subfigure[$\DS{^C_0D_t^{2.5}}x(t)$]{\label{fig_1_2_N_Fixed}\includegraphics[scale=0.5]{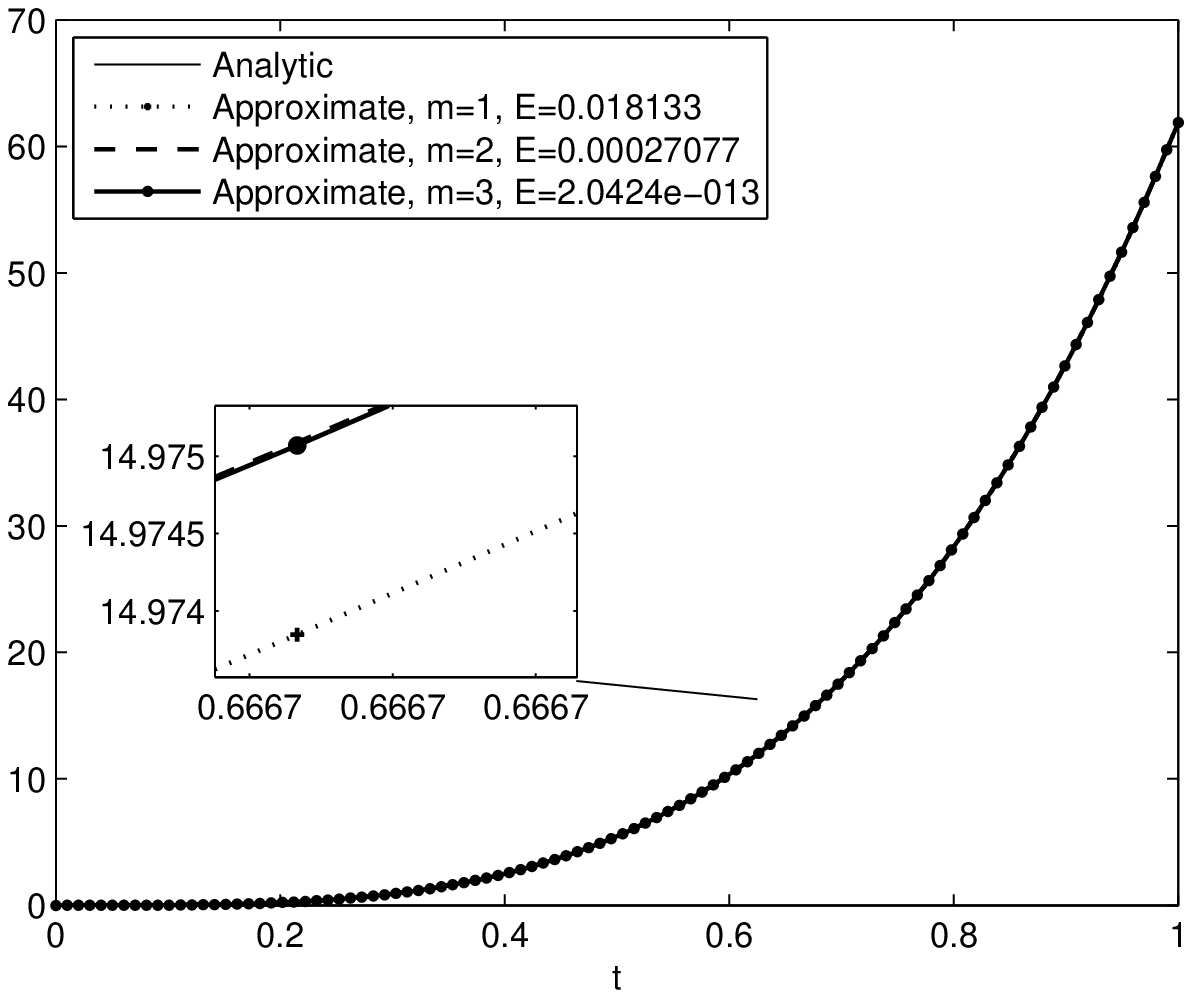}}
        \subfigure[$\DS{^C_tD_1^{1.5}}y(t)$]{\label{fig_1_3_N_Fixed}\includegraphics[scale=0.5]{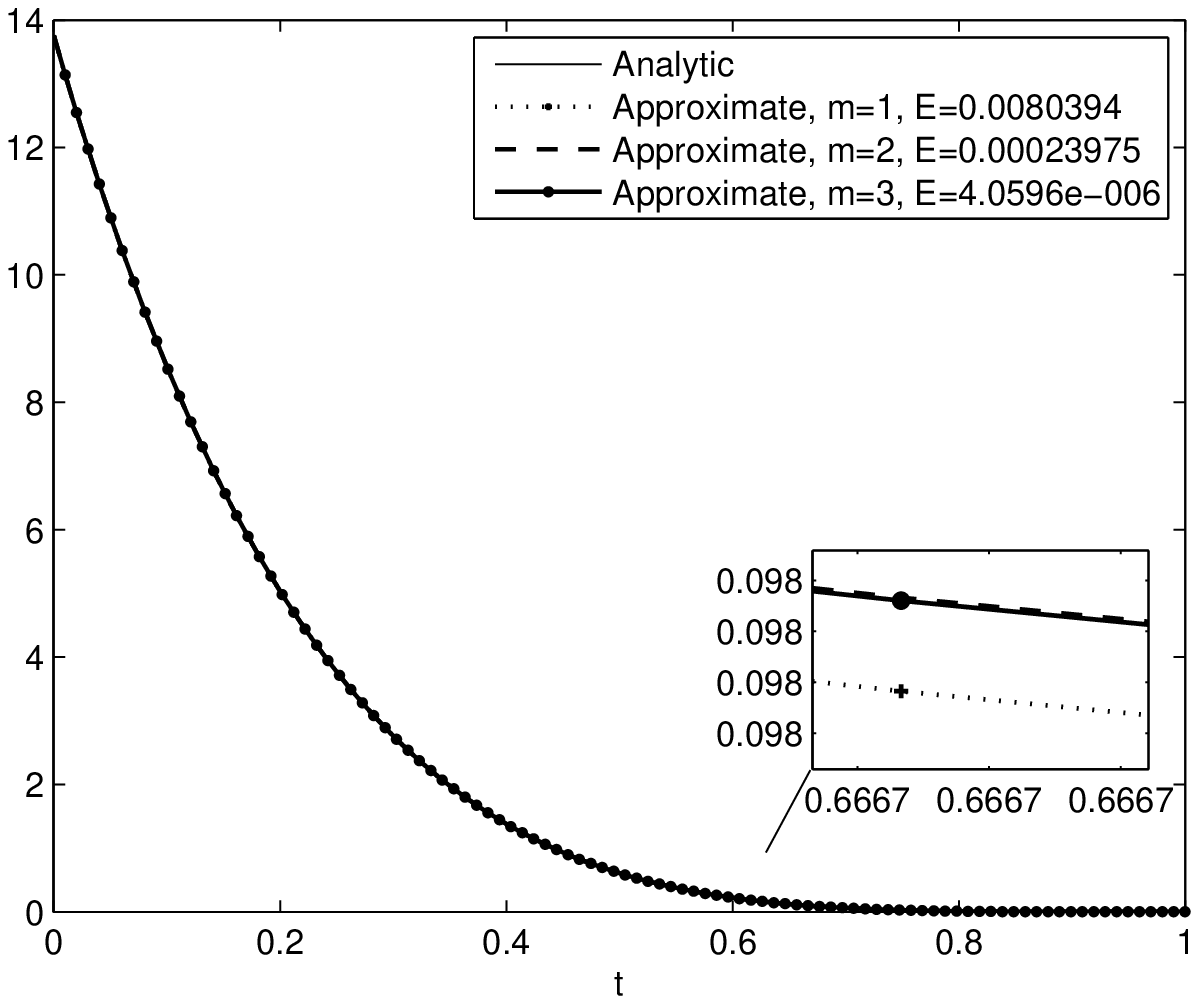}}
        \subfigure[$\DS{^C_tD_1^{2.5}}y(t)$]{\label{fig_1_4_N_Fixed}\includegraphics[scale=0.5]{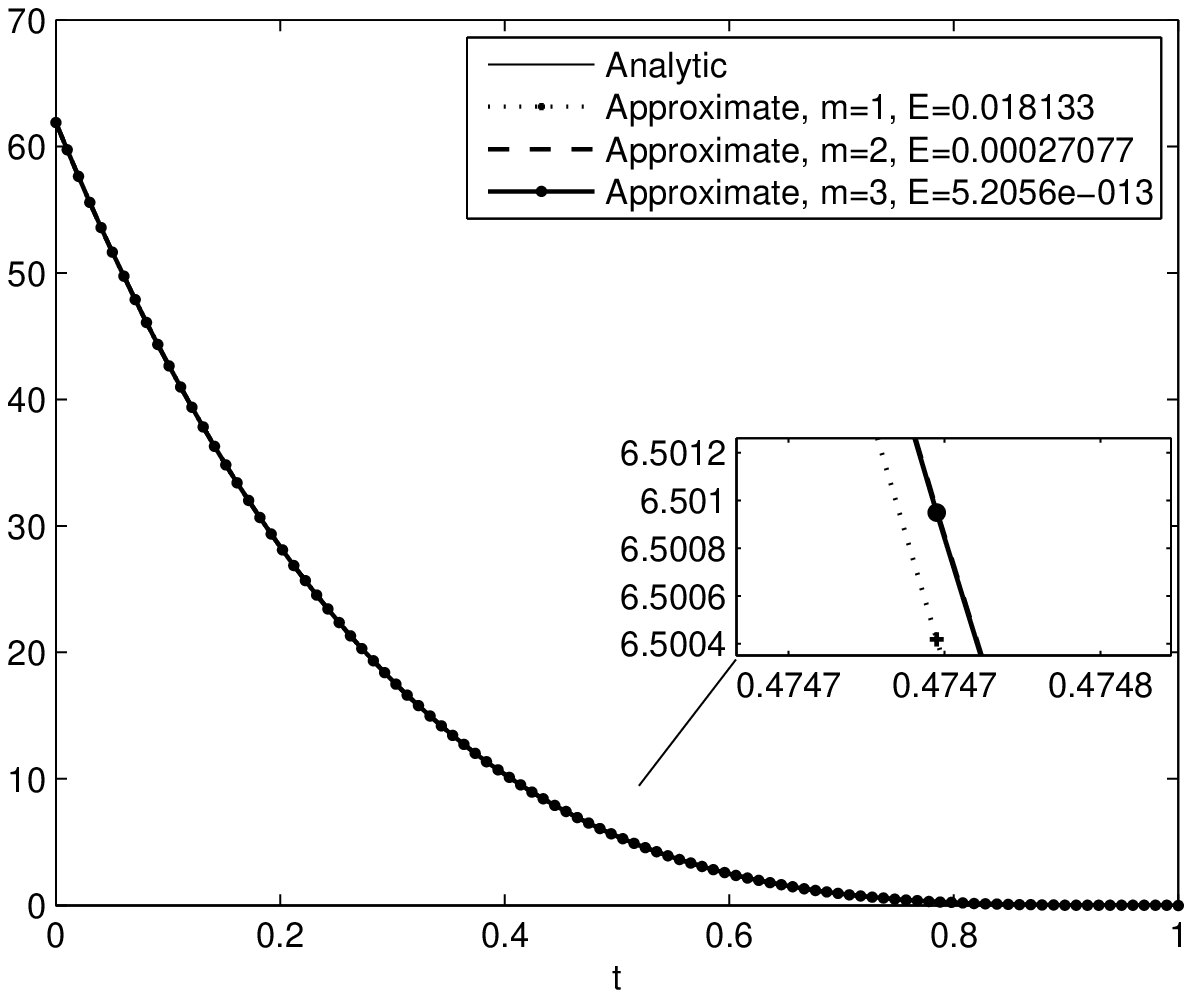}}
  \end{center}
  \caption{For $N=50$: analytic vs. numerical approximation.}\label{N=50}
  \label{N_50}
\end{figure}
\end{example}

\begin{remark} We now compare our approximation given by Theorem \ref{teo1} with the one given by equation \eqref{sousa}. Again, let $x(t)=t^6$ with $t\in[0,1]$, and for the order of the fractional derivative, we consider $\a=1.5$. The result obtained is shown in Figure \ref{fig_sousa} taking $\Delta t=1/100$.

\begin{figure}[h!]
\begin{center}
\includegraphics[scale=0.5]{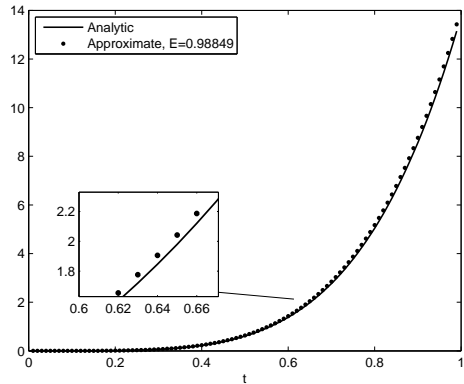}
\end{center}
\caption{For $\Delta t=1/100$: analytic vs. numerical approximation for $\DS{^C_0D_t^{1.5}}x(t)$.}\label{fig_sousa}
\end{figure}
\end{remark}

For our next two examples, we use the MatLab built-in ODE solver \it{ode45}.

\begin{example} Consider a fractional differential equation
$$\left\{
\begin{array}{l}
\DS{^C_0D_t^{2.5}}x(t)+x(t)=\frac{6!}{\Gamma(4.5)}t^{3.5}+t^6, \quad t\in[0,1],\\
x(0)=0,\\
x'(0)=0,\\
x''(0)=0.\\
 \end{array} \right.$$
The obvious solution is $x(t)=t^6$. The idea is to re-write this fractional problem as a system of ordinary differential equations depending only on integer-order derivatives, and after we can apply any numerical tool available to solve it. Since we have three initial conditions, we replace the fractional operator  ${^C_0D_t^{2.5}}x(t)$ by the expansion given in Theorem \ref{teo1}, taking $m=0$, i.e.,
$${^C_0D_t^{2.5}}x(t)\approx A_0t^{0.5} x^{(3)}(t)+\sum_{k=1}^N B_kt^{0.5-k}V_{k-1}(t)$$
with
\begin{equation*}
\begin{split}
A_0&=\DS\frac{1}{\Gamma(1.5)}\left[1+\sum_{p=1}^N\frac{\Gamma(p-0.5)}{\Gamma(-0.5)p!}\right],\\
B_k&=\DS\frac{\Gamma(k-0.5)}{\Gamma(0.5)\Gamma(-0.5)(k-1)!}, \quad k=1,\ldots, N,\\
V_k(t)& \DS=\int_0^t \t^{k}x^{(3)}(\t)d\t, \quad t\in[0,1], \, k=0,\ldots, N-1,
\end{split}
\end{equation*}
and we obtain
$$\left\{
\begin{array}{l}
\DS A_0t^{0.5} x^{(3)}(t)+\sum_{k=1}^N B_kt^{0.5-k}V_{k-1}(t)+x(t)=\frac{6!}{\Gamma(4.5)}t^{3.5}+t^6, \quad t\in[0,1],\\
V_k'(t)=t^{k}x^{(3)}(t), \quad k=0,\ldots, N-1, \quad t\in[0,1],\\
x(0)=0,\\
x'(0)=0,\\
x''(0)=0,\\
V_k(0)=0, \quad k=0,\ldots, N-1.
 \end{array} \right.$$

The result is shown in Figure \ref{FDE1}.
We can see that, as $N$ increases, our numerical approximation becomes closer to the exact solution.

 \begin{figure}[h!]
   \begin{center}
     \includegraphics[scale=0.5]{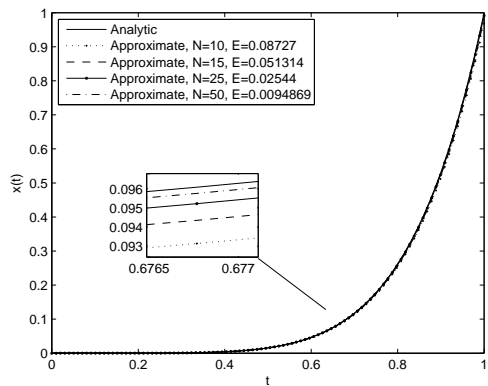}
       \end{center}
   \caption{Example 2: analytic vs. numerical approximation.}\label{FDE1}
\label{fig:ex:2}
 \end{figure}
\end{example}

\begin{example} Consider a fractional coupled mass-spring-damper system with mass value $m$ (in kg), modeled by a fractional differential equation of order $\a\in(1,2)$, for the displacement from the equilibrium position, $x(t)$,
$$m\,  {^C_0D_t^{\a}} x(t)+\gamma \, x'(t)+kx(t)=f(t),$$ with mass initial displacement $x(0)$ and initial velocity $x'(0)$ given, $\gamma$ is the damping coefficient (in N-s/m), $k$ the spring constant  (in N/m) and $f$ is an external force. For simplicity, we will consider $m=\gamma=k=1$, $x(0)=0$, $x'(0)=1$, $\a=1.9$ and $f(t)=\cos(t)$. Thus, the system becomes
\begin{equation}
\left\{
\begin{array}{l}
\DS {^C_0D_t^{1.9}} x(t)+ x'(t)+x(t)=\cos(t),\\
x(0)=0,\\
x'(0)=1.\\
 \end{array} \right. \label{Ex3:exact}
 \end{equation}
 The exact solution for this problem is not known. To solve it numerically, we replace ${^C_0D_t^{1.9}}x(t)$ by the expansion given in Theorem \ref{teo1}. If we consider $m=0$, we obtain the approximated system:
\begin{equation}
\left\{
\begin{array}{l}
\DS A_0t^{0.1} x''(t)+ x'(t)+x(t) +\sum_{k=1}^N B_kt^{0.1-k}V_{k-1}(t)  =\cos(t),\\
V_k'(t)=t^{k}x''(t), \quad k=0,\ldots, N-1,\\
x(0)=0,\\
x'(0)=1,\\
V_k(0)=0, \quad k=0,\ldots, N-1,
 \end{array} \right.  \label{Ex3:Approx:system}
 \end{equation}
with
 $$A_0=\DS\frac{1}{\Gamma(1.1)}\left[1+\sum_{p=1}^N\frac{\Gamma(p-0.1)}{\Gamma(-0.1)p!}\right]$$
 and
 $$B_k=\DS\frac{\Gamma(k-0.1)}{\Gamma(0.1)\Gamma(0.9)(k-1)!}.$$
In Figure \ref{fig:ex:2} we display the results for $t\in[0,20]$.

\begin{figure}[h!]
 \begin{center}
     \includegraphics[scale=0.5]{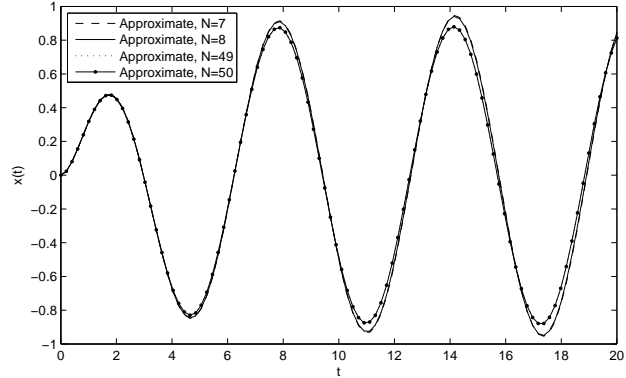}
       \end{center}
\caption{Example 3: analytic vs. numerical approximation.}\label{fig:ex:2}
\end{figure}

Let $g_N$ be the solution of system \eqref{Ex3:Approx:system}. Because the exact solution of system \eqref{Ex3:exact} is not known, we will use as a measure of accuracy, two obtained approximations  $g_{N-1}$ and $g_N$:
$$E(g_{N-1},g_{N})=\left(\sum_{i=1}^{G}(g_{N-1}(t_i)-g_{N}(t_i))^2\right)^{\dfrac12}.$$
 In Table~\ref{EX3:Error:Table} we show some calculated values for different values of $N$. We remark that as $N$ increases, the value decreases.
\begin{table}[!h]
\centering
\begin{tabular}{|c|c|}
\hline  N& $E(g_{N-1},g_N)$ \\
\hline  8&  0.054485696738145 \\
\hline  50&  0.001770846453709\\
\hline
\end{tabular}
\caption{Values for $E(g_{N-1},g_N)$.}\label{EX3:Error:Table}
\end{table}
\end{example}
\rm{
\section{Conclusion}

Fractional differential equations have proven to describe better certain dynamics of real world phenomena, and have called the attention of a vast community of researchers. The drawback is that it is very difficult to deal with them analytically, and so often numerical methods are used to solve the problems. We already find a large number of available methods when the order of the fractional derivative is in the interval $(0,1)$, but is not so common for higher-order derivatives. In this paper we present a general method that can be used to solve fractional differential equations of arbitrary order, by translating the problem into a classical one, depending only on integer-order derivatives. Here, after replacing the fractional operator by the purposed approximation, we solve the problems by discretizing the ordinary differential equations and getting a finite difference equations, and then solve using software Matlab.


\section*{Acknowledgments}

This work was supported by Portuguese funds through the CIDMA - Center for Research and Development in Mathematics and Applications,
and the Portuguese Foundation for Science and Technology (FCT-Funda\c{c}\~ao para a Ci\^encia e a Tecnologia), within project UID/MAT/04106/2013.
We would like to thank the anonymous referee for his/her helpful comments.


}
\end{document}